\documentclass[smallextended]{svjour3}       
\smartqed  
\usepackage{graphicx}
\usepackage{tikz-cd}

\usepackage{amsfonts}
\usepackage{amsmath}
\usepackage{amssymb}

\usepackage{amsthm}
\usepackage[square,numbers]{natbib}

\usepackage{hyperref}

\theoremstyle{definition}

\begin{document}

\title{An algorithm for determining the irreducible polynomials over finite fields}



\author{Samuel H. Dalalyan}


\institute{Yerevan State University\at
              \email{dalalyan@ysu.am}           
}

\date{Received: \today / Accepted: date}

\maketitle

\begin{abstract}
We propose an algorithm for determining the irreducible polynomials over finite fields, 
based on the use of the companion matrix of polynomials
and the generalized Jordan normal form of square matrices.

\end{abstract}




\section{Introduction and the main results}

The problem of finding the irreducible polynomials over  finite fields 
together with the related topic of the irreducible factorization of polynomials 
is one of central themes in the theory of finite fields and in the computational algebra, 
and has numerous applications in coding theory, cryptography, computational number theory. 
There are fairly complete surveys of the work in these areas:  
a comprehensive account of earlier results can be found in the monograph  \citep{LN}, 
while  \citep{GP} provides a survey of relatively recent work.
The best known algorithm for polynomials irreducible factorization is Berlekamp's algorithm \citep{Ber}.
Typically, algorithms for finding irreducible polynomials over a finite field are derived from algorithms 
for irreducible factorization of polynomials 
provided that the irreducible decomposition has a unique factor. 
A well-known algorithm for determining irreducible polynomials over finite fields is 
Rabin's test \citep{Rab}. 
There are other similar algorithms (see, for example, \cite{GaoP}, \cite{Shoup}).

In this paper, we propose an algorithm for finding the irreducible polynomials over finite fields, 
which, in contrast to the aforementioned work,
 is based on the use of the companion matrix of a polynomial, 
the notion of the multiplicative order of a matrix and 
the generalized Jordan normal form of a linear operator (matrix).

The companion matrix of a polynomial 
$$
f(t) = a_0 + a_1 t + a_2 t^2 + ... + a_{d-2} t^{d-2} + a_{d-1} t^{d-1} + t^d,
$$
over a field $\bf F$ is defined as the matrix
$$
[f] =
\left(\begin{array}{cccccc}
0 & 1 & 0 & ... & 0 & 0 \\
0 & 0 & 1 & ... & 0 & 0  \\
0 & 0 & 0 & ... & 0 & 0  \\
...& ... & ... & ... &... & ...  \\
0 & 0  & 0 & ... & 0 & 1 \\
-a_{0} & -a_{1} & -a_{2} & ...  & -a_{d-2} & -a_{d-1}  \\
\end{array} \right).
$$
It is not difficult to check that the characteristic polynomial of the companion matrix 
$[f]$ is equal to $f(t)$. 

The multiplicative order (abbreviated to m.o.) of a square matrix $A$ over a field $\bf F$
is the minimal positive integer $l$ such that $A^l = E$, where $E$ is a unit matrix. 
For the existence of $m.o. A$ it is necessary that $det A \neq 0$. 
This condition is also a sufficient condition, when $\bf F$ is a finite field ${\bf F}_q$, 
where $q = p^n$, $p$ is a prime, $n$ is a positive integer. 
In this paper only finite fields ${\bf F} = {\bf F}_q$ are considered.

The m.o.$[f]$ of the companion matrix $[f]$ of a polynomial $f(t)$ over ${\bf F}_q$
with a non-zero free term $a_0$ is called  {\it the  order} of $f(t)$ 
and is denoted by $ord \, f(t)$ ([1], Definition 3.2 and Lemma 8.26).

Let $m$ be a positive integer, non-multiple of $p$. 
In this (and only in this) case there exist positive integers $c$ such that 
$q^c - 1$ is divided by $m$. 
Any such $c$ is divided by the minimal such integer $e$. 
This $e$ is called the multiplicative order of $q (mod. m)$.

The main results of this paper are the following theorem and the algorithm, 
based on this theorem.
 
\begin{theorem}
\label{th}  
Let ${\bf F} = {\bf F}_q$ be a finite field, 
where $q = p^n$, $p$ is a prime, $n$ is a positive integer. 
Let $f(t)$ be a unitary polynomial over ${\bf F}$ with a non-zero free term, 
$[f]$ be its companion matrix and
\begin{equation}
\label{paramf}
d = deg \, f(t), \quad  m = ord \, f(t) = m.o. [f]. 
\end{equation}
Then the following assertions are true:

1)  $m \leq q^d - 1$ and  $m = q^d - 1$ if and only if $f(t)$ is a primitive irreducible polynomial;

2) more generally, $f(t)$ is an irreducible  polynomial if and only if  
$p$ is not a divisor of $m$,  $d = m.o. q (mod. m)$ and $rk ([f]^l - E) = d$ for all 
positive integers $l <m$, dividing $m$.
\end{theorem}

{\bf Algorithm for finding the irreducible polynomials over a finite field} 

Suppose that $f(t)$ is a unitary polynomial over a finite field ${\bf F} = {\bf F}_q$  of a degree $d$ 
and with a nonzero free term.
\begin{enumerate}
\item Construct the companion matrix $[f]$ of $f(t)$. 

\item  Compute $[f]^l, \, l = 2, 3, ...$ and find $m = m.o. [f]$. 
Then $m \leq q^d - 1$ and $m = q^d - 1$ if and only if
  $f(t)$ is a primitive irreducible polynomial.

\item If $m < q^d - 1$ and $m$ is a multiple of $p$,
then the polynomial $f(t)$ is  reducible.

\item Suppose that $m < q^d - 1$ and $m$ is not a multiple of $p$.
Then compute $q^l (mod. m), \, l = 2, 3, ...$ and find $e = m.o. q (mod. m)$. 
The polynomial $f(t)$ is  reducible  if $e \neq d$. 

\item Suppose that $m < q^d - 1$, $m$ is not a multiple of $p$ and $e = d$.
Then calculate $r = rk ([f]^l - E)$ for divisors $l_1 < l_2 < ...$ of $m$.
If for a divisor $l < m$ of $m$ the rank $r < d$, then $f(t)$ is a reducible polynomial. 
Otherwise, it is irreducible.  
\end{enumerate}

{\bf Remark 1.} 
Calculation of degrees of companion matrix $[f]$ can be realized in a vertical tape (strip) 
as follows.   
Put in the beginning of the strip the matrix $[f]$. 
At every step, assign a new bottom line, equal to the sum 
of the last line, multiplied by $-a_{d-1}$, 
of the penultimate line, multiplied by $-a_{d-2}$, ..., 
of the $d$-th since the end line, multiplied by $-a_0$. 
Then the square matrix, formed by the lines $l, l +1, ..., l +d - 1$
of this strip, is equal to $[f]^l$.

{\bf Remark 2.}
If we know one primitive unitary irreducible over a finite field ${\bf F}_q$ polynomial $f(t)$ of a degree $d$, 
then we can find all unitary irreducible over ${\bf F}_q$ polynomials $g(t)$ of a  degree $d'$, dividing $d$,  
 by formula
\begin{equation}
\label{formula}
g(t)^s = det (tE - [f]^{\frac{q^d-1}{m'}l}), \quad l \leq m', \,  (l, m') = 1,
\end{equation}
where  $d = d's$, $m'$ is a divisor of $q^{d'} - 1$ such that $m.o. q (mod. m') = d'$. 
Then the order $ord \, g(t) = m'$. 

Let us emphasize that by this formula, each unitary irreducible polynomial $g(t)$ of degree $d'$ dividing $d$ 
arises exactly $d'$ times for different values of $l$, satisfying the condition $l \leq m', \,  (l, m') = 1$. 
Thus, we have that the number of unitary irreducible over ${\bf F}_q$ polynomials of a degree $d$ 
and of an order $m$, which are related by the condition $d = m.o. q (mod. m)$,
 is equal to $\varphi(m) \slash d$, where $\varphi(x)$ is the Euler function.

Instead of Theorem $\ref{th}$ we prove the following slightly more general result.

\begin{theorem}
\label{th2}
Let $A$  be a non-singular matrix over a finite field ${\bf F}_q$ 
of size $d \times d$,
$f(t)$ be the characteristic polynomial of $A$, $m = m.o. A$.
Then the following assertions are true. 
\begin{itemize}\itemsep=4pt
\item[\rm 1)] If $f(t)$ is a reducible polynomial, then 
\begin{equation}
\label{inequ0}
m < q^d - 1.
\end{equation}
\item[\rm 2)] The polynomial $f(t)$ is irreducible if and only if 
\begin{itemize}\itemsep=4pt
\item[\rm (a)] $m$ is not divided by $p$,
\item[\rm (b)] $d = m.o. q (mod. m)$,
\item[\rm (c)] $rk (A^l - E) = d$ for all $l < m$, dividing $m$. 
\end{itemize}
\end{itemize}
\end{theorem} 

Theorem $\ref{th}$ is obtained  from Theorem $\ref{th2}$ 
by choosing as matrix $A$ the companion matrix $[f]$ 
of a unitary polynomial $f(t)$ over ${\bf F}_q$. 
Note that the relation $d = m.o. q (mod. m)$ implies that $m$ divides $q^d - 1$ and that
(by definition) $f(t)$ is a {\it primitive} irreducible polynomial if  $m = q^d - 1$.


\section{Preliminary results}

\subsection{Generalized Jordan normal form.} 
Our proof of Theorem $\ref{th2}$ is based 
on a theorem about generalized Jordan normal form 
of a matrix (linear operator) (\cite{F}, \cite{D1}, \cite{D2}).  

According to this theorem any square matrix $A$ 
over a field $\bf F$ is conjugated to a matrix 
$J = P^{-1}AP$, which is a direct sum of 
generalized Jordan blocks: 
\begin{equation}
\label {J}
J = J_1 \oplus ... \oplus J_k.
\end{equation}  
Such a matrix $J$ for any matrix $A$ is determined 
uniquely up to order of direct summands $J_i$
and is called the generalized Jordan normal form of $A$.

Any generalized Jordan block  $J_i$  is a square block-matrix 
of a form
$$
J_i =
\left(\begin{array}{cccccc}
[g_i] & W_i & 0 & ... & 0 & 0 \\
0 & [g_i] & W_i & ... & 0 & 0  \\
0 & 0 & [g_i] & ... & 0 & 0  \\
...& ... & ... & ... &... & ...  \\
0 & 0  & 0 & ... & [g_i] & W_i \\
0 & 0 & 0 & ...  & 0 & [g_i]  \\
\end{array} \right),
$$
where $[g_i]$ is the companion matrix  
of a unitary irreducible over $\bf F$ polynomial $g_i(t)$, 
which is a divisor of the characteristic polynomial $f(t)$ of $A$, 
 $W_i$ is a square matrix with $1$ in the lower left corner 
and $0$ elsewhere.

For example, if $A$ is the companion matrix $[f]$ 
of a unitary polynomial $f(t)$ and 
\begin{equation}
\label{factoriz}
f(t) = g_1(t)^{s_1} ... g_k(t)^{s_k}
\end{equation}
is the irreducible factorization of $f(t)$ with unitary irreducible polynomials $g_i(t)$, 
then $A$ has the generalized Jordan normal form $(\ref{J})$ 
with Jordan blocks $J_i$, having on the main diagonal $s_i$ companion matrices of $g_i(t)$.  

Since under conjugation the multiplicative order, the rank and 
the characteristic polynomial of a matrix are preserved, 
the proof of Theorem $\ref{th2}$ is reduced to the case, 
when $A$ is a Jordan matrix $(\ref{J})$.


\subsection{Some results on block-matrices.} We use the following results on block-matrices. 

The $\bf F$-algebra of square matrices over the field $\bf F$ of size $d \times d$ 
is denoted by $M_d({\bf F})$.
Let $A \in M_d({\bf F})$ and $d = d_1 + ... + d_k$, where $d_1, ..., d_k$ 
are positive integers. 
Below we use the following denotations: 
$$
{\tilde d} = (d_1, ..., d_k), \quad  D_i = d_1 + ... + d_i,  i \in \{1, ..., k \}. 
$$

Let for arbitrary $i, j \in \{ 1, ..., k \}$ 
$A_{ij}$ be the matrix, formed by elements of matrix $A$, 
which are in intersections of lines $D_i - d_i +1, ..., D_i$ 
and rows $D_j - d_j +1, ..., D_j$. 
The matrix $A$ can be considered as a block-matrix with blocks  
$A_{ij}, \, i, j = 1, ..., k$ and $\tilde d$ is called the {\it type} of this block-matrix. 

Let  $A, B \in M_d({\bf F})$ be considered as $\tilde d$-block-matrix 
with blocks $A_{ij}$ and $B_{ij}$, respectively, and  $C = AB$.  
Then the following {\it property of  block multiplication} holds:
the matrix $C$ also can be considered as a $\tilde d$-block-matrix 
with blocks $C_{ij} = A_{i1}B_{1j} + ... + A_{ik}B_{kj}$. 

A $\tilde d$-block-matrix $A$, formed by blocks $A_{ij}$, is called  
{\it upper triangular block-matrix}, 
if all $A_{ij}$ are zero matrices under $i > j$.  

In particular, any generalized Jordan block $J_i$ of size $d_i \times d_i, \, d_i = s_id_i'$ 
can be considered as an upper triangular $\tilde d$-block-matrix with 
${\tilde d}_i = (d_i', ..., d_i')$.

Using the property of block multiplication, one can prove that

1)  the product of two upper triangular $\tilde d$-block-matrices $A$ and $B$ 
with blocks $A_{ij}$ and $B_{ij}$, respectively,
is an upper triangular  $\tilde d$-block-matrix  
and its diagonal blocks are $A_{ii} B_{ii}$; 

2) if $A = A_{11} \oplus ... \oplus A_{kk}$, then 
$A^m = A_{11}^m \oplus ... \oplus A_{kk}^m$ 
for any positive integer $m$; 

3) if  $A = A_{11} \oplus ... \oplus A_{kk}$ and 
each block $A_{ii}$ has a multiplicative order $m_i$, 
then $A$ has the multiplicative order $m = LCM(m_1, ..., m_k)$.


\subsection{Two theorems on the multiplicative order of  matrices over a finite field.}

\begin{theorem}
{\label{th5}}
Let $J$ be a generalized Jordan block of a size $d \times d$,
which has on the main diagonal $s$ companion matrices $[g]$ 
of a unitary irreducible over ${\bf F}_q, q = p^n$ polynomial 
$g(t)$ of  degree $d'$. 
Let $m' = ord \, g(t) = m.o.[g]$ and 
$r$ be the minimal integer such that $p^r \geq d$. 
Then 
\begin{equation}
\label{equ}
m.o.J = p^r m'.
\end{equation}
\end{theorem}

\begin{proof}
Considering $J$ as an upper triangular  block-matrix of the type ${\tilde d} = (d', ..., d')$,
we get that $J^{m'}$ is a (usual) upper triangular matrix of a form $E +N$, 
where $E$ is a unit matrix and 
$N$ is an upper triangular matrix with zeros on the main diagonal, 
having non-zero elements directly after elements of the main diagonal. 
Therefore $N$ is a nilpotent matrix with a nilpotent degree $d$. 

Since the matrices $E$ and $N$ commute, then for any positive integer $k$ we have
$$
(E + N)^k = C_k^0 E^k N^0 + C_k^1 E^{k-1} N^1 + ... + C_k^k E^0 N^k, 
$$
where $C_k^i$ are the binomial coefficients. 
Since matrices $E$ and $N$ are defined over the field ${\bf F}_q$ of characteristic $p$ 
and the binomial coefficients  $C_k^i$ are multiple of $p$ under $i \neq 0, k$, 
when $k = p^r$ for a positive integer $r$,
we get that $(E + N)^{p^r} = E + N^{p^r}$. 
If $p^r \geq d = sd'$,  we have $J^{p^r m'} = (J^{m'})^{p^r} = E$. 

Suppose that $r$ is the minimal integer such that $p^r \geq d$, i. e. 
\begin{equation}
\label{cond}
p^r \geq d, \quad p^{r-1} < d. 
\end{equation}
Then  we obtain the equality ($\ref{equ}$). 
Theorem $\ref{th5}$ is proved. 
\end{proof}

Using all previous results of this section, we get the following theorem.

\begin{theorem}
\label{th4}
Let $A \in M_d({\bf F}_q)$ and $m = m.o. A$. 
Suppose that $A$ is reduced to the generalized Jordan normal form $(\ref{J})$, 
where every generalized Jordan block $J_i$ has on the main diagonal $s_i$  companion matrices 
$[g_i]$ of a unitary irreducible over ${\bf F}_q$ polynomials $g_i(t)$ of a degree $d'_i$ 
and of a multiplicative order $m'_i$.
Suppose that $r_i$ is the minimal integer such that $p^{r_i} \geq d_i = s_i d'_i$, 
if $s_i \geq 2$, and $r_i = 0$, otherwise. 
Put $r = max \{ r_1, ..., r_k \}$. 
Then 
\begin{equation}
\label{mmatr}
m = p^r LCM (m'_1, ..., m'_k).
\end{equation}
\end{theorem}

Note that Theorem 3.11 of \cite{LN} follows immediately from our Theorem $\ref{th4}$, 
if take as matrix $A$  the companion matrix $[f]$ of the polynomial $f(t)$. 
By using the Theorem 3.11 of  \cite{LN}, the inequality $m \leq q^d - 1$ of Theorem $\ref{th}$
of this paper can be obtained from the inequality $p^r LCM (m'_1, ..., m'_k) \leq q^d - 1$.
In subsection 3.2 below this inequality is proved in a more general context of our Theorem $\ref{th2}$.


\subsection{On binomials $t^{q^d} - t$.} For proving the basic results of this paper we also need the following theorem.

\begin{theorem}
\label{th3}
Any binomial 
$$
b_{q,d}(t) = t^{q^d} - t 
$$
is the product  of all unitary irreducible polynomials over the field ${\bf F}_q$ 
of a degree $d'$, dividing $d$, each with the multiplicity one. 
\end{theorem}

\begin{proof}
In its field of decomposition ${\bf F}\left\langle b_{q,d}(t)\right\rangle$, 
the binomial $b_{q,d}(t)$ has exactly $q^d$ different roots, because 
$b_{q,d}'(t) = -1$. 
The set of roots of the binomial $b_{q,d}(t)$ is closed under operations 
of subtraction and division on non-zero elements, 
consequently, this set coincide with the field 
${\bf F}\left\langle b_{q,d}(t)\right\rangle$. 

Let $g(t)$ be an irreducible over ${\bf F}_q$ polynomial of a degree $d'$. 
Then the field ${\bf F}_q[t[ \slash (g(t))$ consists of $q^{d'}$ elements, 
consequently, it is the field of decomposition for the binomial $b_{q,d'}(t)$. 
The following four assertions are equivalent:
\begin{enumerate}
\item[(i)] the polynomial $g(t)$ divides the binomial $b_{q,d}(t)$;
\item[(ii)] the field ${\bf F}\left\langle b_{q,d'}(t)\right\rangle$ is isomorphic 
to a subfield of the field ${\bf F}\left\langle b_{q,d}(t)\right\rangle$;
\item[(iii)] $q^{d'} - 1$ divides $q^d - 1$;
\item[(iv)] $d'$ divides $d$. 
\end{enumerate}  
Therefore the assertions of Theorem $\ref{th3}$ are true.
\end{proof}


\section{Proofs of assertions of Section 1}

\subsection{Proof of  assertion 2 of Theorem $\ref{th2}$.}

 Firstly, starting the proof of Theorem  $\ref{th2}$,
 we suppose that the characteristic polynomial $f(t)$ of a matrix $A \in M_d({\bf F}_q)$ 
with $q = p^n$ is irreducible, and prove that then
$m = m.o. A$ is not a multiple of $p$ and $d = m.o. q (mod. m)$.

If  $f(t) \in {\bf F}_q[t]$ is an arbitrary irreducible polynomial of degree $d$ 
and $(f(t))$ is the prime ideal, generated by $f(t)$, 
then ${\bf F}_q[t] \slash (f(t))$ is a finite field with $q^d$ elements 
and its multiplicative group is a cyclic group of order $q^d - 1$.

Recall that according to Cayley-Hamilton theorem, the matrix $A$ is a root of $f(t)$. 
Since by supposition $f(t)$ is an irreducible polynomial, 
the extension ${\bf F}_q(A)$  of the field ${\bf F}_q$ by $A$
can be considered as a matrix model of the field ${\bf F}_q[t] \slash (f(t))$ 
in the ring of matrices $M_d({\bf F}_q)$. 
Therefore, according to Lagrange theorem, $m = m.o. A$ is a divisor of $q^d - 1$. 
Hence

(i) $m$ is not divided by $p$ and  

(ii) $d$ is an exponent for $q (mod. m)$. 

Since $p$ does not divide $m$, there exists $e = m.o. q (mod m)$, and 
since $d$ is an exponent for $q (mod. m)$, then $e$ divides $d$.

On the other hand, the relation $e = m.o. q (mod. m)$ implies 
that $m$ divides $q^e - 1$. 
If $m = m.o. A$, we get that  $A^{q^e-1} = E$. 
Thus, $A$ is a root of the binomial $b_{q,e}(t)$. 
Since,  according to Cayley-Hamilton theorem,
the irreducible (characteristic) polynomial $f(t)$ 
also  has $A$ as  a root, 
then  $f(t)$ is decomposed in the field 
${\bf F}_q(A) = {\bf F}\left\langle b_{q,d}(t)\right\rangle$ 
on the linear factors and, consequently,
$f(t)$ divides $b_{q,e}(t)$.
 Then, according to Theorem $\ref{th3}$,  $d$ divides $e$. 

Thus, we have proved that $m$ is not divided by $p$ and  $d = e = n.o. q (mod. m)$. 
To finish the proof of the direct implication of the second assertion of Theorem $\ref{th2}$, 
left to prove that $rk (A^l - E) = d$, if  $l < m$ and $l$ divides $m$.

We have that $A^l$ is a root of the binomial $b_{q,d}(t)$.  
Then according to  Theorem $\ref{th3}$ it is a root of 
an irreducible polynomial $g(t)$, dividing $b_{q,d}(t)$, 
with a degree $d' = deg g(t)$, dividing $d$. 

Suppose that $d = d's$.
Then the characteristic polynomial of $A^l$ is equal 
to $g(t)^s$ and  the generalized Jordan normal form of  $A^l$ is 
a direct sum of $s$ copies of companion matrices $[g]$.
Therefore $rk(A^l - E) = d$, if $l < m$.

Now we prove the inverse implication of the second assertion 
of Theorem $\ref{th2}$, i.e. 
the sufficiency the conditions  (a), (b), (c) of Theorem $\ref{th2}$
for irreducibility of the characteristic polynomial  $f(t)$ 
of matrix $A$. 
According to the subsection 2.1, 
it is sufficient to check it for Jordan matrices $\ref{J}$. 

From the condition (a) and Theorems $\ref{th5}$ and $\ref{th4}$ we have that 
each Jordan block $J_i$ is the companion matrix of an irreducible polynomial $g_i(t)$. 

From the condition (c) we get that all $g_i(t)$ have the sane 
multiplicative order $m'$ and, consequently, a same degree 
\begin{equation}
\label{rel1}
d' = m.o. q (mod. m').
\end{equation}
Hence according to $\ref{J}$
\begin{equation}
\label{rel2}
d = d'k
\end{equation}
and according to Theorem $\ref{th4}$ 
\begin{equation}
\label{rel3}
m = m.o. J = m'.
\end{equation}
Then, using the condition (b) and the relations 
$(\ref{rel1}). (\ref{rel2}), (\ref{rel3})$, 
we get that $k = 1$.
The second assertion of Theorem $\ref{th2}$ is proved.


\subsection{Proof of assertion 1 of Theorem $\ref{th2}$.} 

In this subsection we check the inequality ($\ref{inequ0}$).
Actually, it is sufficient to check this inequality 
for generalized Jordan matrices $J = J_1 \oplus ... \oplus J_k$.

If the unitary polynomial $f(t)$ over ${\bf F}_q$ is reducible over this field, then 
either $f(t) = g(t)^s$, where $g(t)$ is a unitary  irreducible polynomial over ${\bf F}_q$ 
and the integer $s \geq 2$, 
or $f(t) = f_1(t) f_2(t)$, where $f_1(t)$ and $f_2(t) $are different unitary 
(therefore non-constant) polynomials.
Correspondingly, there are two following alternative possibilities 
for the Jordan normal form $J$ of $A$: 

(i) $J$ consists of a unique generalized Jordan block, which has 
on the main diagonal $s \geq 2$ companion matrices $[g]$ 
of an irreducible polynomial $g(t)$; 

(ii) $J$ is a direct sum of $k \geq 2$ generalized Jordan blocks. 

We check the inequality $(\ref{inequ0})$ for these two cases separately.

Consider the case (i). Then $J$ is a matrix of  the size $d \times d$, where $d = sd', d' = deg \, g(t)$.

By using Theorem $\ref{th5}$ and  the inequalities\\ 
1) $p^r < d$ of the condition $(\ref{cond})$, \\
2) $m' \leq q^{d'} - 1$, which is true, 
because we have proved in item 2 that 
the multiplicative order $m'$ of the matrix $[g]$ 
with the irreducible over ${\bf F}_q$ characteristic polynomial $g(t)$ 
of order $d'$ divides $q^{d'} - 1$,
( or because $d' = m.o. q (mod. m')$),\\
3) $p \leq q = p^n$,\\
 we obtain that the inequality $(\ref{inequ0})$ follows from 
(is equivalent to) the inequality
\begin{equation}
\label{inequ1}
pd's < (p^{d'})^{s-1} + (p^{d'})^{s-2} + ... + p^{d'} + 1,
\end{equation}

Under $s \geq 3$ this inequality is true, because $px \leq p^x$, if $x = 1, 2$, and  
$p < px \leq p^x$, if $x > 3$.

Under $s = 2$ the inequality ($\ref{inequ1}$), evidently,  also is true, excepting  the cases   

(a) $d' = 1$ and $p$ is an arbitrary prime;

(b) $d' = 2$ and $p = 2$ or $3$; 

(c) $d' = 3$ and $p = 2$. \\

Thus, the inequality $(\ref{inequ0})$ is proved, excepting these three cases.
Below we check that the inequality $(\ref{inequ0})$ holds in these cases as well. 

(a) We have $d' = 1, s = 2, d = sd' = 2$ and from  $(\ref{cond})$ we get that  $r = 1$. 
Therefore the inequality $(\ref{inequ0})$ is reduced to the inequality $p < p + 1$.

(b) $d' = s = 2, d = 4$ and $p$ is $2$ or $3$. 
Then the condition $(\ref{cond})$ implies $r = 2$. 
Consequently, the inequality $(\ref{inequ0})$ take a form 
$p^2 < p^2 + 1$.

(c) $d' = 3, s = 2, d = 6, p = 2$. The condition $(\ref{cond})$ gives $r = 3$ 
and the inequality $(\ref{inequ0})$ take a form $p^3 < p^3 + 1$.

Thus, the inequality $(\ref{inequ0})$ for matrices $J$ 
with a  reducible  characteristic polynomials $f(t)$ 
is proved in the case (i).

Consider the case (ii). 
Let $J$ be a direct sum of $k$ Jordan blocks $J_i$
of sizes $d_i \times d_i$ and of multiplicative orders $m_i$, $i = 1, ..., k$. 
For any Jordan block we have proved above that 
$m.o. J_i = m_i \leq q^{d_i} - 1$. 
Using the block form for $J$ of type ${\tilde d} = (d_1, ..., d_k)$, 
we get that 
$$
J^{m_1 ... m_k} = J_1^{m_1 ... m_k} \oplus ... \oplus J_k^{m_1 ... m_k} = E.
$$ 
Consequently, 
$$
m.o.  J \leq m_1 ... m_k \leq (q^{d_1} - 1)...(q^{d_k} - 1). 
$$
If $s \geq 2$,  we have 
$$
(q^{d_1} - 1)...(q^{d_k} - 1) < q^d - 1, \,\, d = d_1 + ... + d_k
$$
Thus, the first assertion  of theorem $\ref{th2}$ and, consequently, 
all this theorem are completely proved.


\subsection{Proofs of assertions of Remarks 1 and 2.}  

The assertion of Remark 1 follows immediately by the definition 
of multiplication of matrices. 
So, it remains to prove the assertions of Remark 2.

Let $f(t)$ be a unitary primitive irreducible over the finite field ${\bf F}_q$ 
polynomial of a degree $d$.
Then $[f]^l, \, l = 1, ..., q^d - 1$ together with the corresponding zero matrix are 
all roots of the binomial $b_{q,d}(t)$. 
On the other hand, the binomial $b_{q,d}(t) $ is equal to the product 
of all irreducible  over the field ${\bf F}_q$ unitary irreducible polynomials $g(t)$
of a degree $d'$, dividing $d$, with a multiplicity one. 
Therefore  each  such $g(t)$ has $d'$ roots 
$$
[f]^{\frac{q^d-1}{m'}l}, \quad l< m', \, (l, m') = 1
$$
and, naturally, is uniquely  determined by this set of roots. 
Here 
$$
m' = ord \, g(t) =  m.o. [g], \quad
d' = m.o. q (mob. m').
$$ 
Then each such matrix has  a characteristic polynomial 
$g(t)^s$ and its generalized Jordan normal form  is equal to 
the direct sum of $s$ copies of companion matrices $[g]$. 
This explicates Remark 2.


\end{document}